\documentclass[a4paper,10pt]{amsart}
\usepackage{amsfonts}
\usepackage{amssymb}
\usepackage{amsmath}
\usepackage{amsthm}
\usepackage{latexsym}
\usepackage{graphicx}
\usepackage{layout}
 \usepackage{color}
\usepackage{upref}
\usepackage{hyperref}

\newtheorem{theorem}{Theorem} 

\newtheorem{example}{Example}
\newtheorem{proposition}{Proposition}
\newtheorem{lemma}{Lemma}

\newtheorem{remark}{Remark}

 \renewcommand{\(}{\left(}
\renewcommand{\)}{\right)}
\renewcommand{\[}{\left[}
\renewcommand{\]}{\right]}
\newcommand{\eps}{\epsilon}

\newcommand{\rr}{ \mathbb{R}}

\begin{document}
\title[Boundary-layers  for a Neumann problem at  higher critical exponents]{Boundary-layers  for a Neumann problem at higher critical exponents}
 
\author{Bhakti B. Manna}
\address{Bhakti B. Manna, Instituto de Matem\'{a}ticas, Universidad
Nacional Aut\'{o}noma de M\'{e}xico, Circuito Exterior, C.U., 04510 M\'{e}%
xico D.F., Mexico}
\email{mannab@matem.unam.mx}
\author{Angela Pistoia}
\address{Angela Pistoia, Dipartimento di Metodi e Modelli Matematici,
Universit\`a di Roma ``La Sapienza'', via Antonio Scarpa 16, 00161 Roma,
Italy}
\email{angela.pistoia@uniroma1.it}

\subjclass[2010]{35J20,  35J60 }
\keywords{supercritical problem,  blowing-up solutions, boundary  layer}

\begin{abstract}
We consider the  Neumann problem
$$(P)\qquad
- \Delta v +  v= v^{q-1}   \  \text{in }\ \mathcal{D}, \   v > 0 \  \text{in } \ \mathcal{D},\  \partial_\nu v = 0 \ \text{on } \partial\mathcal{D}
,$$
where $\mathcal{D} $ is an open bounded domain in $\rr^N,$     $\nu$ is the unit inner normal at the boundary
and  $q>2.$ 
For any integer,   $1\le h\le N-3,$
we show that, in some
suitable domains $\mathcal D,$  problem $(P)$ has a solution which blows-up along a $h-$dimensional minimal submanifold of the boundary  
$\partial\mathcal D$
as $q$ approaches from either below or above the higher critical Sobolev exponent ${2(N-h)\over N-h-2}.$
\end{abstract}
\maketitle

\section{Introduction}
We are interested  in    the classical Neumann problem
\begin{equation}\label{P}
-\varepsilon^2 \Delta v +  v= v^{q-1}   \  \text{in }\ \mathcal{D}, \quad  v > 0 \quad \text{in } \ \mathcal{D},\quad  \partial_\nu v = 0 \ \text{on } \partial\mathcal{D}
\end{equation}
where $\mathcal{D} $ is an open bounded domain in $\rr^N,$  $\varepsilon $ is a positive parameter,  $q>2 $ and $\nu$ is the unit inner normal at the boundary.
\\

Problem \eqref{P} has deserved a lot of attention in the last decades, because it is a model for different problems in applied science.
It arises for instance as the  shadow system  associated to
activator-inhibitor systems in mathematical theory of biological
pattern formation such as the Gierer-Meinhardt model  \cite{gm} and in
the Keller-Segel model  of chemotaxis \cite{ks}.
A challenging feature of solutions to  \eqref{P}  is that they exhibit {  concentration phenomena} as either the parameter $\varepsilon$ approaches zero or
the esponent {$q$} approaches some critical values.\\

\subsection{The singularly perturbed problem, i.e. $\boldsymbol{\varepsilon\to 0}$}

In the subcritical case, i.e. $q<{2N\over N-2}$  problem \eqref{P} has a least energy solution which  is obtained by minimizing
    the
Rayleigh quotient
\begin{equation}
Q(u) \, =\, {\varepsilon^2 \int_{\mathcal{D}} |\nabla u|^2 + \int_{\mathcal{D}} |u|^2
\over
 (\, \int_{\mathcal{D}} |u|^{q}  )^{2\over q} } ,\quad u\in H^1({\mathcal{D}})\setminus \{0\} ,
\label{Q}\end{equation} for small {$\varepsilon$}. 
In a series of papers Lin, Ni and Takagi {\cite{lin2,nt1,nt2}} proved that
if $\varepsilon$ is   small enough the least energy solution
has a unique local maximum point $\xi_\varepsilon$ which is located on the
boundary $\partial {\mathcal{D}}$ and  approaches as $\varepsilon$ goes to zero the maximum of the mean curvature of the boundary, i.e. $H(\xi_\varepsilon) \to \max_{\xi\in \partial {\mathcal{D}}} H(\xi)$ as $\varepsilon\to0$. Here and in the following
$H$ denotes mean curvature of $\partial{\mathcal{D}}.$   Moreover,
this solution decays exponentially far away from the maximum point which implies indeed the
presence of a very sharp, bounded spike for the solution around
{$\xi_\varepsilon$}.  
Higher energy solutions with similar qualitative
behavior have been found by several authors: 
solutions with boundary peaks  in  {\cite{dy,DFW,g1,gww,li,w2}},
with interior peaks in  \cite{bf,cw,gp,gpw,gw1,k,w1,y} or with both  boundary and interior peaks in  \cite{gw2}.
In particular, we quote the result   in \cite{w2}, where the author proved that such a spike solution
exists around any non-degenerate critical point of the mean curvature.
\\

Phenomena of this type occurs as well in the critical case
$q=\frac{2N}{N-2}$, however several important differences are
present. For instance, since compactness of the embedding of
$H^1(\mathcal D)$ into $L^{q}(\mathcal D)$ is lost, existence of
minimizers of $Q(u)$ becomes non-obvious.
 (and in general not true
for large $\varepsilon$ as   established in \cite{Lin}). It is the
case however, as shown in \cite{am1,wz1}, that such a minimizer does
exist if $\varepsilon$ is sufficiently small. The
profile and asymptotic behavior of this least energy solution has
been analyzed in \cite{apy1,apy2,npt,r5}. Again the least energy solution has only one local maximum
point   located around a point of maximum mean curvature of
$\partial\Omega.$ Unlike the subcritical case, the $L^\infty-$norm of the least energy solution is unbounded as $\eps$ goes to zero.
 Construction of solutions with this
type of   blowing-up behavior  around one or more critical points
of the mean curvature has been achieved for instance in
{\cite{am2,amy,g,gg,gl,lww,r2,r5,wxj,wz1,wz2,wz3,wy}}. An important difference with the
subcritical case is that now mean curvature is required to be
positive at these critical points. Indeed, non-negativity of
curvature is actually necessary for existence
\cite{gl,r5}. Moreover, in contrast with the subcritical situation, there are no solutions which   blow-up only in interior points, i.e. at least one blow-up point has to lie on the boundary as established in \cite{cny,r2}.

\subsection{The almost critical problem, i.e. $\boldsymbol {q\to{2N\over N-2}}$ if $\boldsymbol {N\ge3}$ or $ \boldsymbol  {q\to+\infty}$ if $  \boldsymbol   {N=2}$}

  The first result in this direction is obtained in  \cite{cn}, where the authors  proved that if $N\ge 4$ and if   the exponent $q$ 
approaches the critical exponent {\em from below}, i.e. $q={2N\over N-2}+\epsilon$ with $\epsilon$ small and negative, then
 there exists a    solution blowing-up at  
points located on the boundary  near critical points of the mean
curvature with {\em negative value}.  

Recently, Rey and Wei in \cite{rw2} and del Pino, Musso and Pistoia \cite{dmp} proved that  if $N\ge 4$ and if   the exponent {$q$} 
approaches the critical exponent {\em from above}, i.e. $q={2N\over N-2}+\epsilon$ with $\epsilon$ small and positve, then
 there exists a   solution  blowing-up at
points located on the boundary near critical points of the mean
curvature with {\em positive value}.
The case $N=3$ is more involved and it has been treated by Rey and Wei in \cite{rw1}.

If $N=2$ the exponent $q$ is allowed to go to $+\infty$ and Musso and Wei in \cite{mw} constructed    solutions concentrating in  interior and  boundary points, whose location is  determined by a suitable combination of  Green's functions.

\subsection{Concentration along higher dimensional sets}

It is natural to look for solutions to problem \eqref{P} that exhibit concentration phenomena
   not just at points but on higher dimensional subsets of $\overline {\mathcal D}$ as suggested by Ni in \cite{n}.\\

It is useful to introduce the {\em $h-$th critical exponent.} Given an integer $h $ we define
$$q^*_{h}:={2(N-h)\over N-h-2}\ \hbox{if}\ 0\le h\le N- 3\quad\hbox{and}\quad q^*_{N-2}:=+\infty\ \hbox{if}\  h=N-2.$$
$q^*_{h}$ is nothing but the critical Sobolev exponent of the embedding $H^1(\Omega)\hookrightarrow L^{q^*_h}(\Omega)$ where $\Omega$ is a smooth 
bounded domain in $R^{N-h}.$
We also note that $q^*_0={2N\over N-2}.$

In particular, given  $\Gamma$ a $h-$dimensional submanifold of 
the boundary and assuming that   $1\le h\le N-3$
and $q\le q^*_h,$ the question is  whether there exists a solution to \eqref{P} which concentrates along $\Gamma$ as $\varepsilon\to0.$

In \cite{mama,mamo1,mamo2,ma} the authors have established the existence of such a  solution (for a suitable sequence of parameters 
$\varepsilon_i\to0$) in the $h-$subcritical case, i.e. $q<q^*_h$, when either $\Gamma$ is the whole boundary  or $\Gamma$ is an embedded closed minimal submanifold of $\partial\mathcal D$
which is in addition nondegenerate in the sense that its Jacobi operator is nonsingular. 

Recently, in \cite{dmm} the result has been extended to the $h-$critical case, i.e. $q=q^*_h$.  The authors assume that $\Gamma$ is an embedded closed minimal
submanifold of $\partial\mathcal D$ whose dimension is $h\le N-7$ (in particular, $N\ge8$) which is nondegenerate, and a certain weighted average 
of sectional curvatures of the boundary $\partial\mathcal D$ is 
positive along $\Gamma$. Then they  prove the existence of a solution to problem \eqref{P} for a suitable sequence of parameters 
$\varepsilon_i\to0$ which blows-up along $\Gamma.$
\\

As far as we know, there are not any results about existence of this kind of solutions to problem \eqref{P} when the parameter $\varepsilon$ 
is fixed (say $\varepsilon=1$)
and the exponent $q$ is allowed to approaches the higher critical exponents $q^*_h.$ In particular, a natural question arises

\begin{itemize}
\item[(Q)]
   {\em for any integer $h=1,\dots,N-3$, if $q $ approaches   $q^*_h$, does problem 
   \eqref{p} have a positive solution  which blows-up along a suitable $h-$dimensional minimal  submanifold of the boundary 
   $\partial\mathcal{D}?$}
  \end{itemize}
  
In the present paper,     we give a positive answer   when the domain   $\mathcal{D}$ has some rotational symmetry.

\subsection{Our result}

 Let $n\ge1 $ and $n\ge m\ge1$ be fixed integers.   Let $\Omega $ be a smooth open bounded domain in $\rr^n $ such that
$$\overline {\Omega} \subset \{\left( x_{1},\ldots ,x_{m},x^{\prime
}\right) \in \mathbb{R}^{m}\times \mathbb{R}^{n-m}:x_{i}>0,\text{ }i=1,\ldots
,m\}.
$$
Let $M=M_{1}+\cdots +M_{m},$ $M_{i}\geq 2,$ and set
\begin{equation*}
\mathcal{D}:=\{(y_{1},\dots ,y_{m},x^{\prime })\in \mathbb{R}^{M_{1}}\times
\cdots \times \mathbb{R}^{M_{m}}\times \mathbb{R}^{n-m}:\left( \left\vert
y_{1}\right\vert ,\ldots ,\left\vert y_{m}\right\vert ,x^{\prime }\right)
\in \Omega\}.
\end{equation*}%
Then $\mathcal{D}$ is a smooth bounded domain in $\mathbb{R}^{N}$, $N:=M+n-m.$
 Set $h:=M-m,$ so that $N-h=n.$

The solutions we are looking for are $\mathcal{G}-$invariant for the action of
the group $\mathcal{G}:=\mathcal{O}(M_1)\times\dots\times\mathcal{O}(M_m)$ on $\rr^N$ given by
\begin{equation*}
(g_{1},\dots ,g_{m})(y_{1},\dots ,y_{m},x^{\prime }):=(g_{1}y_{1},\dots
,g_{m}y_{m},x^{\prime }).
\end{equation*}%
Here $\mathcal{O}(M_i)$ denotes the group of linear isometries of $\rr^{M_i}.$

A simple calculation shows that a function $v$ of the form $v(y_{1},\dots
,y_{m},x^{\prime })=u\left( \left\vert y_{1}\right\vert ,\ldots ,\left\vert
y_{m}\right\vert ,x^{\prime }\right) $ solves problem 
\begin{equation} \label{p}
-\Delta v+v= v^{q-1}\  \text{in}\ \mathcal D,\qquad  v>0\  \text{in}\
\mathcal D,\qquad \partial_\nu v=0\  \text{on}\
\partial \mathcal D
\end{equation}%
if $u$ solves
\begin{equation} \label{p11}
-\Delta u+\sum_{i=1}^{m}\frac{M_{i}-1}{x_{i}}\frac{\partial u}{\partial
x_{i}}+u= u^{q-1}\  \text{in}\ \Omega,\qquad  u>0\  \text{in}\
 \Omega,\qquad \partial_\nu u=0\  \text{on}\
\partial \Omega.
\end{equation}%
Here  $q=2^*_h={2n\over n-2}+\epsilon $ where  $\epsilon$ is a small {\em positive} or {\em negative}  parameter. 
Thus, we are lead to study the more general anisotropic problem
\begin{equation}
-\text{div}(a(x)\nabla u)+a(x)u=a(x)u^{q-1}\  \text{in}\ \Omega,\qquad  u>0\  \text{in}\
 \Omega,\qquad \partial_\nu u=0\  \text{on}\
\partial \Omega,\label{p1}
\end{equation}
where $\Omega $ is a bounded smooth domain in $\mathbb{R}^{n},$  $a\in C^{2}(\overline{\Omega })$ and 
$\min\limits_{x\in \overline{\Omega }}a(x)>0.
$
Note that if
\begin{equation}\label{a}
a(x_{1},\ldots ,x_{m}):=x_{1}^{M_{1}-1}\cdots x_{m}^{M_{m}-1},
\end{equation}
problem \eqref{p1} reduces to problem \eqref{p11}.

Our goal is to construct solutions to problem \eqref{p1}   which blow-up at a suitable critical point $\xi _{0}$
of $a$ constrained on the boundary  $\partial\Omega$ as $\epsilon$ goes to $0.$ It corresponds
to construct solutions to  problem \eqref{p}    which   blow-up along the $\mathcal{G}-$orbit  $\Xi(\xi_0)$ of $\xi _{0}$ lying on the boundary $\partial\mathcal D$  as  $\epsilon$ goes to $0.$ Here
\begin{equation}\label{orbit}
\Xi(\xi_0):=\{(y_{1},\dots,y_{m},x^{\prime })\in \partial D:\left( \left\vert
y_{1}\right\vert ,\dots,\left\vert y_{m}\right\vert ,x^{\prime }\right)=\xi_0
\in \partial\Omega\}
\end{equation} 
  is a $h-$dimensional  minimal submanifold  of the boundary of $\mathcal{D} $
  diffeomorphic to $ \mathbb{S}^{M_1-1}\times\dots,\times\mathbb{S}^{M_m-1}  $ (recall that $M-m=h$),
 where $\mathbb{S}^{M_i-1}$ is the unit sphere in $\rr^{M_i} . $

Set
\begin{equation}\label{ha}{\mathcal H_a(\xi)}:= {2\over n-1}{\partial_\nu a(\xi) \over a(\xi)}
- H(\xi)  .\end{equation}

 Our main result is the following.

\begin{theorem}
\label{main1}  Let $N-h\ge5.$  Assume $\xi_0$ is a  $C^1$-stable critical point of $a$ constrained on the boundary.
\begin{itemize}
\item[$(i)$] 
If $\mathcal H_a(\xi_0)<0$ there
exists $\epsilon _{0}>0$ such that for each $\epsilon \in (0,\epsilon _{0})$
problem \emph{(\ref{p})} has
     a positive solution $v_{\epsilon }$ which
blows-ip along $\Xi(\xi_0)$ (see \eqref{orbit}) as $\epsilon \rightarrow 0.$
\item[$(ii)$]  
If $\mathcal H_a(\xi_0)>0$ there
exists $\epsilon _{0}>0$ such that for each $\epsilon \in ( -\epsilon _{0},0)$
problem \emph{(\ref{p})} has
     a positive solution $v_{\epsilon }$ which
blows-up along $\Xi(\xi_0)$ (see \eqref{orbit}) as $\epsilon \rightarrow 0.$
\end{itemize}\end{theorem}
 
 We recall that any strict local maximum point or strict local minimum point or non-degenerate critical point is a  $C^1$-stable critical point of $a$ (see for instance \cite{li}).
\\ 

\begin{remark}We point out that   the sign of $\mathcal H_a(\xi_0)$ is nothing but   the sign of a weighted average of sectional curvatures of the submanifold $\Xi(\xi_0)$ (see \eqref{orbit}).  Indeed a simple calculation shows that if $a$ is given as in \eqref{a} then
$${\partial_\nu a(\xi) \over a(\xi)}= {M_1-1\over x_1}\nu_1+ \dots+{M_m-1\over x_m}  \nu_m,$$
where $\nu_1,\dots,\nu_m$ are the first $m-$components of the unit inner normal $\nu$ at the boundary point $\xi.$
A similar   weighted average of sectional curvatures was introduced in (1.13) of   \cite{dmm}.\end{remark}

According to the previous discussion, Theorem \ref{main1} is an immediate consequence of the following result.

\begin{theorem}
\label{main} Let $n\ge5.$
 Assume $\xi_0$ is a  $C^1$-stable critical point of $a$ constrained on the boundary.
\begin{itemize}
\item[$(i)$] 
If $\mathcal H_a(\xi_0)<0$  there exists $\epsilon _{0}>0$ such that, for each $\epsilon
\in (0,\epsilon _{0}),$ problem \emph{(\ref{p1})} has a
solution $u_{\epsilon }$  which blows-up at $\xi _0$ as $\epsilon\to0.$
  \item[$(ii)$] 
If $\mathcal H_a(\xi_0)>0$ there exists $\epsilon _{0}>0$ such that, for each $\epsilon
\in ( -\epsilon _{0},0),$ problem \emph{(\ref{p1})} has a
solution $u_{\epsilon }$which blows-up at $\xi _0$ as $\epsilon\to0.$
 \end{itemize}\end{theorem}

It is useful to give a simple example.

\begin{example}
Let $\Omega$ be a strict convex domain in  $\mathbb R^{N-1}$, so that $H(\xi)>0$ for any $\xi\in\partial\Omega .$  
Let $D$ the torus-like domain 
\begin{equation*}
\mathcal{D}:=\{(y_{1},  x^{\prime })\in \mathbb{R}^{2}\times
 \mathbb{R}^{N-2}:\left( \left\vert
y_{1}\right\vert , x^{\prime }\right)
\in \Omega\}.
\end{equation*} In particular $h=1,$   $n=N-1 $ ($m=1$ and $M=2$) and  $a(x)=x_1$. In this case the exponent $q$ approaches $q^*_1={2(N-1)\over N-3}.$ It is clear that $a$ constrained on $\partial\Omega$ has two  $C^1$-stable critical points: a strict minimum point $\xi_{\mathfrak m} $
and a strict maximum point $\xi_{\mathfrak M} $, i.e.
$$ a(\xi_{\mathfrak m} )  :=\min\limits_{\xi\in\partial\Omega} \xi_1\quad \hbox{and}\quad  a(\xi_{\mathfrak M} )  :=\max\limits_{\xi\in\partial\Omega} \xi_1.$$
It is immediate to check that
 $$\mathcal H_a(\xi_{\mathfrak m}) =c_N\({2\over N-2}{1\over a(\xi_{\mathfrak m} )}-H(\xi_1)\)\ \hbox{and}\ \mathcal H_a( \xi_{\mathfrak M} ) =c_N\(-{2\over N-2} {1\over a(\xi_{\mathfrak M} )}-H(\xi_2)\).$$
 In particular, $\mathcal H_a(\xi_{\mathfrak M})<0.$ On the other hand, $\mathcal H_a(\xi_{\mathfrak m})<0$ if $a(\xi_{\mathfrak m} )$ is large enough while $\mathcal H_a( \xi_{\mathfrak m} )>0$ if $a(\xi_{\mathfrak m} )$ is small enough.
 \\
 It is interesting to note that in this case problem \eqref{P} has  always a solution when the exponent $q$ approaches $2^*_{ 1}$ from above, which concentrate along the geodesic $\Xi(\xi_{\mathfrak M})$. On the other hand, if the exponent $q$ approaches $2^*_{ 1}$ from below existence of solutions to problem \eqref{P} depends on the size of the {\em hole} of $\mathcal D$, i.e. if  $a(\xi_{\mathfrak m} )$ is small enough   then there exists a solution which concentrate along the geodesic $\Xi(\xi_{\mathfrak m})$.
\end{example}

\medskip
The proof of our results relies on a very well known Ljapunov-Schmidt reduction.
We omit many details of the finite dimensional reduction, because they   can be found, up to some minor modifications, in the literature.  
 In Section \ref{uno}  we sketch  the main steps of the proof  and we prove Theorem \ref{main}. The proofs which can not be immediately deduced from known results are given in the Appendix.

\section{Scheme of the Proof}\label{uno}

\subsection{Setting of the problem}
We introduce the Hilbert space $\mathrm{H} ^{1}(\Omega )$ equipped with the inner product and the corresponding norm
$$(u,v):=\int_{\Omega }a(x)\nabla u  \nabla v\text{ }dx\ \hbox{and}\
\Vert u\Vert :=\left( \int_{\Omega }a(x)\left\vert \nabla u\right\vert
^{2}dx\right) ^{1/2}.
$$
 We also introduce the Banach space  $\mathrm{L}^{r}(\Omega ),$ $r\in
\lbrack 1,\infty )$, equipped with the norm
\begin{equation*}
|u| _{r}:=\left( \int_{\Omega }a(x)\left\vert u\right\vert
^{r}dx\right) ^{1/r}.
\end{equation*}%

  Let $i^{\ast }:%
\mathrm{L}^{\frac{2n}{n+2}}(\Omega )\rightarrow \mathrm{H} ^{1}(\Omega )$
be the adjoint operator of the embedding $i:\mathrm{H} ^{1}(\Omega
)\hookrightarrow \mathrm{L}^{\frac{2n}{n-2}}(\Omega ),$ i.e. $i^{\ast }(f)=u$
if and only if%
\begin{equation*}
-\text{ div}(a(x)\nabla u)+a(x) u=a(x)f\quad \text{in}\ \Omega ,\qquad \partial _\nu u=0\quad
\text{on}\ \partial \Omega .
\end{equation*}
It is clear that there exists a positive constant $c$ such that
\begin{equation}\label{i1}
\left\Vert i^{\ast }(f)\right\Vert \leq c|f| _{\frac{2n}{%
n+2}}\quad \forall \ f\in \mathrm{L}^{\frac{2n}{n+2}}(\Omega ).
\end{equation}%

To study problem \eqref{p1} in the supercritical case, it is useful to recall this regularity result (see for example {\cite{mp})}. 
inequality
\begin{equation}\label{i2}
\left|i ^{*}(f)\right|_{s}\le C\left|f\right|_{\frac{ns}{n+2s}}\text{ for }f\in L^{\frac{ns}{n+2s}}
\end{equation}
 where $s\ge\frac{2n}{n-2} $ so that $\frac{ns}{n+2s}\ge\frac{2n}{n+2}$.

Next, we consider the Banach space ${H} _\epsilon=H ^{1}(\Omega)\cap L ^{s_{\epsilon}}(\Omega)$
with the norm
$$
\|u\|_{H_\epsilon}=\|u\| +\left|u\right|_{s_{\epsilon}}
$$
where we set
$ s_{\epsilon}= {2n\over n-2}+\epsilon \frac{n}{2}
$ if $\epsilon>0$ and $s_\epsilon={2n\over n-2}$ if $\epsilon\le 0.$
We remark that if $\epsilon\le 0$ the space
$ {H}_{\epsilon}$ is nothing but the space $H ^{1}(\Omega)$ with the norm
$\|\cdot\| $.
Finally,  also using  the maximum principle, problem (\ref{p1}) is equivalent to the problem
\begin{equation}
u=i ^{*}( f_{\epsilon}(u))\ \ \ u\in {H}_{\epsilon}\label{eq1}
\end{equation}
where $f_{\epsilon}(u)=(u^{+})^{p+\epsilon},$  
$u^{+}=\max\{u,0\}$ and $p:=\frac{n+2}{n-2}.$

\subsection{The approximated solution}

Let us introduce the {\em standard bubbles}
\begin{equation*}
U_{\delta ,\xi }:=\alpha_n {\frac{\delta ^{\frac{%
n-2}{2}}}{(\delta ^{2}+|x-\xi |^{2})^{\frac{n-2}{2}}},}\qquad \delta >0,%
\text{\quad }\xi \in \mathbb{R}^{n},
\end{equation*}
with $\alpha _{n}:=\left[ n(n-2)\right] ^{\frac{n-2}{4}},$
which are the positive solutions to the limit problem
\begin{equation*}
-\Delta u=u^p ,\qquad u\in H^{1}(\mathbb{R}^{n}).
\end{equation*}%
It is useful to recall that the set of solutions to the linearized problem
\begin{equation*}
-\Delta Z =pU_{\delta ,\xi }^{p-1}Z \ \hbox{in}\ \mathbb{R}^{n}.
\end{equation*}%
is spanned by the functions
\begin{equation*}
 Z _{\delta ,\xi }^{0}(x):={\frac{\partial U_{\delta ,\xi }}{\partial
\delta }}=\alpha _{n}{\frac{n-2}{2}}\delta ^{\frac{n-4}{2}}{\frac{|x-\xi
|^{2}-\delta ^{2}}{(\delta ^{2}+|x-\xi |^{2})^{n/2}}}
\end{equation*}%
and, for each $j=1,\dots ,n,$
\begin{equation*}
Z _{\delta ,\xi }^{j}(x):={\frac{\partial U_{\delta ,\xi }}{\partial \xi
_{j}}}=\alpha _{n}(n-2)\delta ^{\frac{n-2}{2}}{\frac{x_{j}-\xi _{j}}{(\delta
^{2}+|x-\xi |^{2})^{n/2}}}.
\end{equation*}%

Let $PU_{\delta ,\xi }$ denote the solution of the problem
onto $\mathrm{H} ^{1}(\Omega )$, i.e.
\begin{equation*}
-\Delta PU_{\delta ,\xi }+PU_{\delta ,\xi }=U_{\delta ,\xi }^p\ \text{\ in}\ \Omega ,\qquad \partial_\nu PU_{\delta ,\xi }=0\ \text{\ on}\ \partial
\Omega .
\end{equation*}

 We look for a solution to problem (\ref{eq1}) as
$$
u_{\epsilon }= PU_{\delta  ,\xi  } +\phi_\epsilon ,
$$
where the concentration point   and the concentration parameter satisfy
$$\xi \in\partial\Omega\quad\hbox{and}\quad
 \delta  =|\epsilon|  d\quad \hbox{for some}\
d >0.  $$

The rest term $\phi_\epsilon$ belongs to a suitable space defined as follows.
Let us introduce the spaces
$$
K_{d,\xi} :=\mathrm{span}\{Z _{\delta ,\xi }^j\ :\ j=0,1,\dots ,n\}$$ and
$$
K_{d,\xi} ^{\perp }:=\left\{ \phi \in \mathrm{H}%
 ^{1}(\Omega ):(\phi ,Z _{\delta ,\xi }^{j})=0,\  j=0,1,\dots ,n\right\},$$
and  the  projection operators
\begin{equation*}
\Pi _{d,\xi}(u):= \sum%
\limits_{j=0}^{n}(u,Z _{\delta ,\xi }^j)Z _{\delta ,\xi }^j\quad \hbox{and}\quad \Pi _{d,\xi}^{\perp }(u):=u-\Pi _{d,\xi}(u).
\end{equation*}

As usual, our approach to solve problem \eqref{eq1} will be to find a $(d,\xi)\in \mathbb R\times \partial\Omega $ and a function $\phi \in K_{d,\xi}^{\perp }$ such that
\begin{equation}
\Pi _{d,\xi}^{\perp }\left\{ PU_{\delta,\xi}+\phi -i^{\ast }\left[ f_{\epsilon }\(PU_{\delta,\xi}+\phi\)\right] \right\} =0  \label{es1}
\end{equation}%
and
\begin{equation}
\Pi _{d,\xi} \left\{ PU_{\delta,\xi}+\phi -i^{\ast }\left[ f_{\epsilon }\(PU_{\delta,\xi}+\phi\)\right] \right\} =0.
  \label{es2}
\end{equation}

\subsection{Reduction to a finite dimensional problem: solving equation \eqref{es1}}
First, we   find for any $(d,\xi)\in \mathbb R\times \partial\Omega 
$ and small $\epsilon $ a function $\phi \in K_{d,\xi}^{\perp }$ such that \eqref{es1} holds. To this aim we define a
linear operator $L_{d,\xi}:K_{d,\xi}^{\perp }\rightarrow K_{d,\xi
}^{\perp }$ by
\begin{equation}\label{elle}
L_{d,\xi}\phi :=\phi -\Pi _{d,\xi}^{\perp }i^{\ast }\left[ f^{\prime }_\epsilon(PU_{\delta,\xi})\phi \right] .
\end{equation}%
We will prove the following.

\begin{proposition}
\label{pro1} For any compact subset $C $ of $\mathbb R\times \partial\Omega  $ there exist $%
\epsilon _{0}>0$ and $c>0$ such that for each $\epsilon \in (-\epsilon_0,\epsilon
_{0}) $ and $(d,\xi)\in C $ the operator $%
L_{d,\xi}$ is invertible and
\begin{equation*}
\left\Vert L_{d,\xi}\phi \right\Vert _{H_\epsilon} \geq
c\left\Vert \phi \right\Vert _{H_\epsilon}\ \quad \ \forall \ \phi \in K_{d,\xi}^{\perp }.
\end{equation*}
\end{proposition}

\begin{proof}
We argue as in Lemma 3.1 of \cite{mp}.
\end{proof}

Now, we are in position to solve equation \eqref{es1}.

\begin{proposition}
\label{pro2} For any compact subset $C $ of $\mathbb R\times \partial\Omega $ there exist $%
\epsilon _{0}>0$ and $c>0$ such that for each $\epsilon \in (-\epsilon_0,\epsilon
_{0}) $ and $(d,\xi)\in C $ there exists
a unique $\phi _{d,\xi}^{\epsilon }\in K_{d,\xi}^{\perp }$ which solves \eqref{es1} and satisfies
$$\left\Vert \phi _{d,\xi}^{\epsilon }\right\Vert _{H_\epsilon}
\leq c|\epsilon||\ln\epsilon|. $$
Moreover, the map $(d,\xi)\to \phi _{d,\xi}^{\epsilon }$ is a $C^1-$map which satisfies
$$
\left\Vert D_{(d,\xi)}\phi _{d,\xi}^{\epsilon }\right\Vert _{H_\epsilon}
\leq c|\epsilon||\ln\epsilon|.  
$$
\end{proposition}

\begin{proof}
The proof is postponed to Appendix. 
\end{proof}

\subsection{The reduced problem: solving equation \eqref{es2}}
We now introduce the energy functional $J_{\epsilon }:\mathrm{H}\rightarrow \mathbb{R}$ defined by
\begin{equation*}
J_{\epsilon }(u):={\frac{1}{2}}\int\limits_{\Omega }a(x)\(|\nabla u|^{2}+u^2(x)\)dx-{%
\frac{1}{p+1+\epsilon }}\int\limits_{\Omega }a(x)(u^+)^{p+1+\epsilon }dx,
\end{equation*}%
whose critical points are the solutions to problem \eqref{eq1}. Let us define
the reduced energy functional $\widetilde{J}_{\epsilon }:\mathbb R\times\partial\Omega \rightarrow
\mathbb{R}$ by
\begin{equation*}
\widetilde{J}_{\epsilon }(d,\xi):=J_{\epsilon
}(PU_{\delta,\xi}+\phi _{d,\xi}^{\epsilon })
\end{equation*}%
Next, we prove that the critical points of $\widetilde{J}_{\epsilon }$ are
the solutions to problem \eqref{es2}.

\begin{proposition}
\label{pro3} The function $PU_{\delta,\xi}+\phi _{d,\xi}^{\epsilon }$ is a critical point of the
functional $J_{\epsilon }$ if and only if the point $(d,\xi)$ is a critical point of the function $\widetilde{J}_{\epsilon }.$
\end{proposition}

\begin{proof}
We argue as in Proposition 2.2 of \cite{mp}.\end{proof}

The problem is thus reduced to finding   critical points of $\widetilde{%
J}_{\epsilon } $ and so it is necessary to compute the asymptotic expansion of $%
\widetilde{J}_{\epsilon }$.

\begin{proposition}
\label{pro4}It holds true that
\begin{align*}
& \widetilde{J}_{\epsilon }(d,\xi)=a(\xi)\left[
c_{1}+c_{2} \epsilon \log |\epsilon|+c_3\epsilon  +c_4\mathcal H_a(\xi)|\epsilon|d +c_5\epsilon\ln d+o(\epsilon)\right],
\end{align*}%
$C^{1}$-uniformly on compact sets of $\mathbb R\times \partial\Omega .$
Here $\mathcal H_a(\xi)$ is defined in \eqref{ha}, 
  $c_{i}$  are constants and in particular, $c_{4}$ and $c_{5} $ are positive.
\end{proposition}

\begin{proof}
The proof is postponed to Appendix.
\end{proof}

\subsection{Proof of Theorem \ref{main}}.

We apply Proposition \ref{pro3}. Then it is enough to prove that the reduced energy $\widetilde{J}_{\epsilon }$ has a critical point if $\epsilon$ is small enough.
By Proposition \ref{pro4} we deduce that
\begin{equation*}
\nabla _\xi \widetilde{J}_{\epsilon }(d,\xi)=c_1 \nabla _\xi a(\xi)+o(1).
\end{equation*}
Moreover
\begin{equation*}
\nabla _d \widetilde{J}_{\epsilon }(d,\xi)=\[c_4 \mathcal H_a(\xi)   +c_5{1\over d}+o(1)\]\epsilon \ \hbox{if}\ \epsilon>0,
\end{equation*}
and
\begin{equation*}
\nabla _d \widetilde{J}_{\epsilon }(d,\xi)=\[-c_4{\mathcal H_a(\xi)}   +c_5{1\over d}+o(1)\]\epsilon \ \hbox{if}\ \epsilon<0.
\end{equation*}
Let $\xi_0$ be a $C^1$-stable critical point of $a$ constrained on the bondary.
By Brouwer degree theory it follows that if either
$d_0:=-{c_5\over c_4}{1\over \mathcal H_a(\xi_0) }$ if $\epsilon>0$ or $d_0:= {c_5\over c_4}{1\over \mathcal H_a(\xi_0) }$ if $\epsilon<0,$ if $\epsilon$ is small enough there exists
$(d_\epsilon,\xi_\epsilon)$ such that
 $\nabla_{(d,\xi)}\widetilde{J}_{\epsilon }(d_\epsilon,\xi_\epsilon)=0$, $d_\epsilon\to d_0$ and $\xi_\epsilon\to \xi_0$ as $\epsilon\to0.$ That proves our claim.

 \appendix
 \section{}
\subsection{A local parametrization of the boundary}
 Take $\xi\in\partial\Omega$. Without loss of generality we may take $\xi=0$ and the unit inward normal of $\partial\Omega$ at $\xi$   directed along the $x_n$-axis.
Denote $x'=(x_1,x_2,\dots,x_{n-1})$, $B'(0,r)=\{x'\in\mathbb R^{n-1}:|x'|<r\}$, and $\Omega _1:=\Omega\cap B(\xi,r)$, where 
$B(\xi,r)=\{x\in\mathbb R^{n }:|x-\xi|<r\}$. Since $\partial\Omega$ is smooth, we can find a constant $r>0$ such that $\partial\Omega\cap B(\xi,r)$ can be 
represented by the graph of a smooth function $\rho_\xi : B'(0,r)\to \mathbb R$, where $\rho_\xi(0)=0$ and $\nabla \rho_\xi(0)=0$, and
$$ \Omega\cap B(\xi,r)=\{(x',x_n)\in B(\xi,r) : x_n>\rho_\xi(x')\}.$$ Moreover, we may write
$$ \rho_\xi(x')=\frac{1}{2}\sum_{i=1}^{n-1}k_ix_i^2+O(|x|^3)$$ where $k_i, i=1,2,\dots,n-1$, are the principle 
curvatures at $\xi$. The mean curvature is defined by
$$ H(\xi) = \frac{1}{n-1}\sum_{i-1}^{n-1}k_i$$

\subsection{The expansion of the ansatz}
From Lemma A.1 of \cite{rw2} we have that  
\begin{equation}\label{espu1} P U_{\delta,\xi}(x)=  U_{\delta,\xi}(x)-\delta^{\frac{4-n}{2}}\varphi_0\(\frac{x-\xi}{\delta}\)+O\(\delta^{\frac{6-n}{2}} \)\end{equation}
where   $\varphi_0$ solves

 \begin{equation}
  \label{E.IstApprox}
\left\{\begin{aligned}
      \Delta \varphi_0&=0  &&\text{in } \mathbb R^n_+\\
      \frac{\partial \varphi_0}{\partial x_n} &=\alpha_n \frac{n-2}{2}\sum_{i=1}^{n-1}\frac{k_i x_i^2}{(1+|x|^2)^{n/2}} &&\text{on }\partial R^n_+\\
      \varphi_0&\to 0 &&\text{ as } |x|\to\infty
    \end{aligned}
  \right.
\end{equation}

Using Green's expression $\varphi_0$ can be expressed as 
$$\varphi_0(x)=\alpha_n
\frac{1}{\omega_{n-1}}\sum_{i=1}^{n-1}k_i\int_{\mathbb R^{n-1}}\frac{y_i^2}{(1+|y'|^2)^{n/2}} \frac{1}{|x-y'|^{N-2}}dy'$$
where $\omega_{n-1}$ is the surface measure of the unit sphere in $\mathbb R^n$. Moreover
\begin{equation}
 |\varphi_0(x)| \le\frac{C}{(1+|x|)^{n-3}}.\label{espu2}
\end{equation}

 \subsection{Proof of Proposition \ref{pro2}}

 As it is usual  equation \eqref{es1} turns out to be  equivalent to
\begin{equation*}
L_{d,\xi}(\phi)=N_{d,\xi}(\phi)+R_{d,\xi},\end{equation*}
where the linear operator $L_{d,\xi}$ is defined in \eqref{elle}, the error-term is
\begin{equation}\label{re}
 R_{d,\xi}:=\Pi^\perp_{d,\xi}\left\{i^* \left[ f _\eps\left(PU_{\delta,\xi}\right)\right]-PU_{\delta,\xi}\right\}
 \end{equation}
and the higher order term is 
\begin{equation*}
 N_{d,\xi}(\phi):=\Pi^\perp_{d,\xi}\left\{
 i^* \[  f_\eps \left(PU_{\delta,\xi}+\phi\right)-f_\eps \left(PU_{\delta,\xi}\right)-f'_\eps \left(PU_{\delta,\xi} \right)\phi 
\right] \right\}
 \end{equation*}
By Proposition \ref{pro1}, using Lemma \ref{re1} and the usual contraction mapping argument, the claim follows
(see, for instance, Proposition 2.1 of \cite{mp}).

We recall the following useful lemma.
\begin{lemma}\label{yyl}
For any $a>0$ and $b\in\mathbb R$ we have
$$\left||a+b|^\beta- a^\beta\right|\le\left\{ 
\begin{aligned}&c(\beta)\min\{|b|^\beta,a^{\beta-1}|b|\}\ \hbox{if}\ 0<\beta<1\\
&c(\beta)\(|b|^\beta+a^{\beta-1}|b|\)\ \hbox{if}\ \beta>1\\
\end{aligned}
\right.$$
and
$$\left||a+b|^\beta(a+b)-a^{\beta+1}-(1+\beta)a^\beta b\right|\le\left\{ 
\begin{aligned}&c(\beta)\min\{|b|^{\beta+1},a^{\beta-1}b^2\}\ \hbox{if}\ 0<\beta<1\\
&c(\beta)\max\{|b|^{\beta+1},a^{\beta-1}b^2\}\ \hbox{if}\ \beta>1\\
\end{aligned}
\right.$$
\end{lemma}

Let us estimate the error term \eqref{re}.

\begin{lemma}\label{re1}
It holds  
 $$  \left\|  R_{d,\xi}\right\| _{H_\epsilon}=O\(|\eps|\ln|\eps|\).$$
\end{lemma}
\begin{proof}
 
 By the definition of $i^*,$ we immediately get  that
 $$PU_{\delta,\xi}=i^*\left[f_0(U_{\delta,\xi})-{\nabla a(x)\over a(x)}\nabla PU_{\delta,\xi}\right].$$
 Then  by \eqref{i1} and \eqref{i2} we get
\begin{align}\label{R0}
 \left\|  R_{d,\xi}\right\|_{H_\epsilon}=&O\( |f_\eps(P U_{\delta,\xi})-f_0(PU_{\delta,\xi}) |_{ 2n \over n+2 }+ |f_\eps(P U_{\delta,\xi})-f_0(PU_{\delta,\xi}) |_{ ns_\epsilon \over n+2s_\epsilon}\) \nonumber \\ &+O\(  |f_0(P U_{\delta,\xi})-f_0(U_{\delta,\xi}) |_{ 2n \over n+2 }+|f_0(P U_{\delta,\xi})-f_0(U_{\delta,\xi}) |_{ ns_\epsilon \over n+2s_\epsilon}\) \nonumber \\ &+
O\( |{\nabla a(x)\over a(x)}\nabla PU_{\delta,\xi} |_{2n\over n+2}+ |{\nabla a(x)\over a(x)}\nabla PU_{\delta,\xi} |_{{ ns_\epsilon \over n+2s_\epsilon}}\) \nonumber \\ &=:I_1+I_2+I_3 .\end{align}
Since ${ ns_\epsilon \over n+2s_\epsilon}={2n\over n+2}+O(\epsilon),$ we only compute the $L^{2n\over n+2}-$norms.
  Using the same arguments of  Proposition 2 of \cite{r1} we can estimate $I_1$ as
  \begin{align}\label{I1}
  |f_\eps(P U_{\delta,\xi})-f_0(PU_{\delta,\xi}) |_{ 2n \over n+2 }=O\(|\epsilon\ln\epsilon|\).\end{align}
Let us estimate $I_2.$
By Lemma \ref{yyl}, \eqref{espu1} and \eqref{espu2} we deduce that
 \begin{align}\label{I2} |f_0(P U_{\delta,\xi})-f_0(U_{\delta,\xi}) |_{ 2n \over n+2 }&=O\( | P U_{\delta,\xi} -U_{\delta,\xi}  |^p_{ 2n \over n-2 }\)+O\(  U_{\delta,\xi}^{p-1}\(P U_{\delta,\xi} -U_{\delta,\xi} \) | _{ 2n \over n+2 }\)\nonumber\\ &= O\( \delta \) \hbox{(because $n\ge5$).}
 \end{align}
Finally, we estimate $I_3 $ as
 \begin{align}\label{I3}
&\left|{\nabla a(x)\over a(x)}\nabla PU_{\delta,\xi} \right|_{2n\over n+2}= O\( \delta \)\ \hbox{(because $n\ge5$).}\end{align}
By \eqref{R0}, \eqref{I1}, \eqref{I2} and \eqref{I3} the claim follows.

\end{proof}
 \subsection{Proof of Proposition \ref{pro4}}
 
First of all it is quite standard to prove that (see, for instance, Proposition 2.2 of \cite{mp})
\begin{equation}\label{j1}J_{\epsilon
}(PU_{\delta,\xi}+\phi _{d,\xi}^{\epsilon })=J_{\epsilon
}(PU_{\delta,\xi})+o(\epsilon).\end{equation}
Therefore, we only have to estimate $J_{\epsilon
}(PU_{\delta,\xi}).$
We remark that
\begin{equation}\label{j2}\begin{aligned}
J_{\epsilon
}(PU_{\delta,\xi})=&a(\xi)\[{1\over2}\int\limits_\Omega\(|\nabla u|^2+u^2\)dx-{1\over p+1+\epsilon}\int\limits_\Omega (u^+)^{p+1+\epsilon }dx\]\\
&+\int\limits_\Omega\[a(x)-a(\xi)\]\[{1\over2}\(|\nabla u|^2+u^2\) -{1\over p+1+\epsilon} (u^+)^{p+1+\epsilon }\]dx\\
:=&I_1+I_2
\end{aligned}\end{equation}
$I_1$ was estimated in Proposition A.1 of \cite{rw2} (note that in our case the expansion in \cite{rw2} also contain $\alpha_n^2$) as
\begin{equation}\label{j3}
I_1=a(\xi)\[c_1+c_2\epsilon\ln|\epsilon|+c_3 \epsilon +c_5\epsilon\ln d- c_6  d|\epsilon|H(\xi)+o(\epsilon)\]\end{equation}
where  $c_i$ are constants. In particular $c_5$ is positive and 
$$c_6:={(n-2)^2\over n-3}\alpha_n^2\int\limits_{\mathbb R^{n-1}}{|y'|^2\over (1+|y'|^2)^n}dy'$$
Let us estimate $I_2$ taking into account that 
$$a(x)=a(\xi)+\nabla a(x)(x-\xi)+O\(|x-\xi|^2\).$$
Then we have 
\begin{equation}\label{j4}\begin{aligned}I_2&=\int\limits_\Omega\[\nabla a(x)(x-\xi)+O\(|x-\xi|^2\)\]\[{1\over2}\(|\nabla u|^2+u^2\) -{1\over p+1+\epsilon} (u^+)^{p+1+\epsilon }\]dx\\
&\hbox{(scaling $x-\xi=\delta y$)}\\
&=\delta \[\nabla a(\xi)\int\limits_{\mathbb R^n_+} y\[{1\over2} |\nabla U|^2  -{1\over p+1 } U^{p+1 }\]dy+o(1)\]\\
&=\delta \[\partial_\nu a(\xi)\int\limits_{\mathbb R^n_+} y_n\[{(n-2)^2\over2}\alpha_n^2 {|y|^2\over (1+|y|^2)^n} -{1\over p+1 } \alpha_n^{p+1 }{1\over (1+|y|^2)^n}\]dy+o(1)\]\\
&=c_7 d |\eps| \partial_\nu a(\xi) +o(\epsilon),
\end{aligned}\end{equation}
where
$$c_7:={(n-2)^2\over2}\alpha_n^2 \int\limits_{\mathbb R^n_+} y_n{|y|^2-1\over (1+|y|^2)^n}dy.$$
We need to compare the two constants $c_6$ and $c_7.$
First of all, we have (since $n\ge5$)
$$\int\limits_{\mathbb R^n_+} y_n{|y|^2-1\over (1+|y|^2)^n}dy={1\over 2(n-1)(n-2)}\int\limits_{\mathbb R^{n-1}} {1\over (1+|y'|^2)^{n-2}}dy'
+{1\over 2(n-1)}\int\limits_{\mathbb R^{n-1}} {|y'|^2-1\over (1+|y'|^2)^{n-1}}dy'$$
For any positive real numbers $p$ and $q$ such that $p-q>1$, we let
\begin{equation*}
I^q_p=\int_0^{+\infty}\frac{r^q}{\(1+r\)^p}dr\,=2\int_0^{+\infty}\frac{s^{2q+1}}{\(1+s^2\)^p}ds\,={\Gamma(q +1)\Gamma(p-q -1)\over\Gamma(p)},
\end{equation*}
where $\Gamma$ is the Gamma Eulero function.
In particular, we have
\begin{equation*}
I^q_{p+1}=\frac{p-q-1}{p}I^q_p\quad\text{and}\quad I^{q+1}_{p+1}=\frac{q+1}{p-q-1}I^q_{p+1}\,.
\end{equation*}
 Then we can write
 $$c_6 ={(n-2)^2\over 2(n-3)}\alpha_n^2\omega_{n-2}I^{n-1\over2}_n. $$
 where $\omega_{n-2}$ is the measure of the $(n-2)-$dimensional unit sphere.
 An easy computation leads to
$$\int\limits_{\mathbb R^{n-1}} {1\over (1+|y'|^2)^{n-2}}dy'={1\over 2}\omega_{n-2}I^{n-3\over2}_{n-2}= {2(n-2)\over n-3}\omega_{n-2}I^{n-1\over2}_{n },$$
$$\int\limits_{\mathbb R^{n-1}} {|y'|^2\over (1+|y'|^2)^{n-1}}dy'={1\over 2}\omega_{n-2}I^{n-1\over2}_{n-1}= {n-1\over n-3}\omega_{n-2}I^{n-1\over2}_{n }$$
and
$$\int\limits_{\mathbb R^{n-1}} {1\over (1+|y'|^2)^{n-1}}dy'={1\over 2}\omega_{n-2}I^{n-3\over2}_{n-1}= \omega_{n-2}I^{n-1\over2}_{n }$$
Then
$$c_7={ (n-2)^2 \over (n-1)(n-3) }\alpha_n^2\omega_{n-2}I^{n-1\over2}_n.$$
By collecting, the previous estimates, we have that
\begin{equation}\label{j5}- c_6 d|\epsilon|a(\xi)H(\xi)+c_7 d |\eps| \partial_\nu a(\xi) ={ (n-2)^2 \over 2( n-3 )}\alpha_n^2\omega_{n-2}I^{n-1\over2}_n d |\epsilon|a(\xi)\( {2\over n-1}{\partial_\nu a(\xi) \over a(\xi)}
- H(\xi)\).\end{equation} 

By \eqref{j1}, \eqref{j2}, \eqref{j3}, \eqref{j4} and \eqref{j5} the claim follows.

\section*{Acknowledgments} The first author is supported by a postdoctoral fellowship under CONACYT grant 237661 (Mexico). The second author is supported by GNAMPA and Sapienza Fondi di Ricerca.

\end{document}